\documentclass[11pt]{amsart}
\usepackage{geometry,graphicx}
\usepackage{labelfig}

\usepackage{epsfig}
\usepackage{epstopdf}

\usepackage{amsfonts}
\usepackage{amsmath}
\usepackage{amsthm}
\usepackage[all]{xy}
\usepackage{latexsym}
\usepackage{enumerate}
\usepackage{amssymb}
\usepackage[cp850]{inputenc}
\usepackage{amsfonts}
\usepackage{graphicx}

\newcommand{\BN}{\mathbb N}			\newcommand{\BQ}{\mathbb Q}
			\newcommand{\BZ}{\mathbb Z}

\newcommand{\CC}{\mathcal C}

		\newcommand{\CP}{\mathcal P}
		
		\newcommand{\CT}{\mathcal T}

\newcommand{\DP}{{\mathcal {DP}}}
\newcommand{\CCP}{{\mathcal {CP}}}
\newcommand{\DZ}{{\mathcal{D}\mathbb Z}}

\newcommand{\teichmuller}{Teichm\"uller }

\newtheorem*{namedtheorem}{\theoremname}
\newcommand{\theoremname}{testing}
\newenvironment{named}[1]{\renewcommand{\theoremname}{#1}\begin{namedtheorem}}{\end{namedtheorem}}

\newtheorem{conjecture}{Conjecture}[section]

\newtheorem{corollary}[conjecture]{Corollary}

\newtheorem{lemma}[conjecture]{Lemma}

\newtheorem{proposition}[conjecture]{Proposition}

\newtheorem{remark}[conjecture]{Remark}
\newtheorem{theorem}[conjecture]{Theorem}

\begin{document}

\title[]{Convexity of strata in diagonal pants graphs of surfaces}
\author{J. Aramayona, C. Lecuire, H. Parlier \& K. J. Shackleton}

\address{School of Mathematics, Statistics and Applied Mathematics, 
 NUI Galway, Ireland}
\email{javier.aramayona@nuigalway.ie}
\urladdr{http://www.maths.nuigalway.ie/$\sim$javier/}

\address{Laboratoire Emile Picard
Universit\'e Paul Sabatier, Toulouse, France}
\email{lecuire@math.ups-tlse.fr}
\urladdr{http://www.math.univ-toulouse.fr/$\sim$lecuire/}

\address{Department of Mathematics, University of Fribourg, Switzerland}
\email{hugo.parlier@unifr.ch}
\urladdr{http://homeweb.unifr.ch/parlierh/pub/}

\address{University of Tokyo IPMU\\ Japan 277-8583}
\email{kenneth.shackleton@impu.jp}
\urladdr{http://ipmu.jp/kenneth.shackleton/}

\begin{abstract}
We prove a number of convexity results for strata of the diagonal pants graph of a surface, in analogy with the extrinsic geometric properties of strata in the Weil-Petersson completion. As a consequence, we exhibit convex flat subgraphs of every possible rank inside the diagonal pants graph. 
\end{abstract}

\maketitle

\section{Introduction}

Let $S$ be a connected orientable surface, with empty boundary and negative Euler characteristic. The {\em pants graph}
$\CP(S)$ is the graph whose vertices correspond to homotopy classes of pants decompositions of $S$, and where two vertices are adjacent
if they are related by an elementary move; see Section \ref{prelim} for an expanded definition. The graph $\CP(S)$ is connected, and becomes a geodesic metric space by declaring each edge to have length 1.

A large part of the motivation for the study of $\CP(S)$ stems from the result of Brock \cite{brock} which asserts that  $P(S)$ is quasi-isometric to $\CT(S)$, the \teichmuller space of $S$ equipped with the Weil-Petersson metric. As such, $P(S)$ (or any of its relatives also discussed in this article) is a combinatorial model for Teichm\"uller space.

By results of Wolpert \cite{wolpert} and Chu \cite{chu}, the space $\CT(S)$ is not complete. Masur \cite{masur}  proved that its completion
${\hat \CT}(S)$ is homeomorphic to the {\em augmented \teichmuller space} of $S$, obtained from  $\CT(S)$ by extending Fenchel-Nielsen
coordinates to admit zero lengths. The completion ${\hat \CT}(S)$ admits a natural {\em stratified structure}: each stratum $\CT_C(S) \subset {\hat \CT}(S)$ corresponds to a multicurve $C \subset S$, and parametrizes surfaces with {\em nodes} exactly at the elements of $C$. 
 By Wolpert's  result \cite{wolpert2} on the convexity of length functions, $\CT_C(S)$ is convex in  ${\hat \CT}(S)$ for all multicurves $C \subset S$.

The pants graph admits an analogous stratification, where the stratum $\CP_C(S)$ corresponding to the multicurve $C$ is the subgraph of $\CP(S)$ spanned by those pants decompositions that contain $C$.
Moreover, Brock's quasi-isometry between $\CP(S)$ and $\CT(S)$ may be chosen so that the image of $\CP_C(S)$ is contained in $\CT_C(S)$. 

In light of this discussion, it is natural to study which strata of $\CP(S)$ are convex. This problem was addressed in \cite{APS1, APS2}, where certain families of strata in $\CP(S)$ were proven to be totally geodesic; moreover,
it is conjectured that this is the case for all strata of $\CP(S)$, see Conjecture 5 of \cite{APS1}. As was observed by Lecuire, the validity of this conjecture is equivalent to the existence of only finitely many geodesics between any pair of vertices of $\CP(S)$; we will give a proof of the equivalence of these two problems in the Appendix.

The main purpose of this note is to study the extrinsic geometry of strata in certain graphs of pants decompositions closely related to $\CP(S)$, namely the {\em diagonal pants graph} $\DP(S)$ and the {\em cubical pants graph}.  Concisely,  $\DP(S)$ (resp. $\CCP(S)$)  is obtained from $\CP(S)$ by adding an edge of length 1 (resp. length $\sqrt{k}$) between any two pants decompositions that differ by $k$ {\em disjoint} elementary moves.  Note that Brock's result \cite{brock} implies that both $\DP(S)$ and $\CCP(S)$ are quasi-isometric to $\CT(S)$, since they are quasi-isometric  to $\CP(S)$. These graphs have recently arisen in the study of metric properties of moduli spaces; indeed, Rafi-Tao \cite{RaTa} use $\DP(S)$ to estimate the  Teichm\"uller diameter of the thick part of moduli space, while $\CCP(S)$ has been 
used by Cavendish-Parlier to give bounds on the Weil-Petersson diameter of moduli space.

As above, given a multicurve $C\subset S$, denote by $\DP_C(S)$ (resp. $\CCP_C(S))$ the subgraph of $\DP(S)$  (resp. $\CCP(S)$) spanned by those pants decompositions which contain $C$. Our first result is:

\begin{theorem}\label{thm:punctures}
Let $S$ be a sphere with punctures and $C \subset S$ a multicurve. Then $\DP_C(S)$ 
is convex in $\DP(S)$. 
\end{theorem}

We remark that, in general, strata of $\DP(S)$ are not totally geodesic; see Remark \ref{rmk:outside3} below.

The proof of Theorem \ref{thm:punctures} will follow directly from properties of the {\em forgetful maps} between graphs of pants decompositions. In fact, the same techniques will allow us to prove the following more general result:

\begin{theorem}
Let $Y$ be an essential subsurface of a surface $S$ such that $Y$ has the same genus as $S$.
Let $C$ be the union of a pants decomposition of $S \setminus Y$ with all the boundary
components of $Y$. Then:

\begin{enumerate}
\item $\DP_C(S)$ is convex in $\DP(S)$.
 \item If $Y$ is connected, then $\CP_C(S)$ and $\CCP_C(S)$ are totally geodesic inside $\CP(S)$ and $\CCP(S)$, respectively.
\end{enumerate}
\label{main2}
\end{theorem}

Next, we will use an enhanced version of the techniques in \cite{APS1} to show an analog of Theorem \ref{thm:punctures} for
general surfaces, provided the multicurve $C$ has {\em deficiency} 1; that is, it has one curve less than a pants decomposition. 
Namely, we will prove:

\begin{theorem}\label{thm:general}
Let $C\subset S$ be a multicurve of deficiency $1$. Then:
\begin{enumerate}
\item  $\DP_C(S)$
is convex in $\DP(S)$. 
\item  $\CCP_C(S)$ 
is totally geodesic in $\CCP(S)$. 
\end{enumerate}
\end{theorem}

We observe that  part (2) of Theorem \ref{thm:general} implies the main results in \cite{APS1}. 

\medskip

We now turn to discuss some applications of our main results. The first one concerns the existence of convex flat subspaces of $\DP(S)$ of every possible rank. Since $\DP(S)$ is quasi-isometric to $\CP(S)$, the results of  Behrstock-Minsky \cite{bemi}, Brock-Farb \cite{brfa} and Masur-Minsky \cite{MM2} together yield that $\DP(S)$ admits a quasi-isometrically embedded copy of $\BZ^r$ if and only if $r\le [\frac{3g+p-2}{2}]$. If one considers ${\hat T}(S)$ instead of $\DP(S)$, one obtains {\em isometrically} embedded copies of $\BZ^r$ for the exact same values of $r$. In \cite{APS2}, convex copies of $\BZ^2$ were exhibited in $\CP(S)$, and it is unknown whether higher rank convex flats appear. The corresponding question for the diagonal pants graph is whether there exist convex copies of $\BZ^r$ with a modified metric which takes into account the edges added to obtain $\DP(S)$; we will denote this metric space $\DZ^r$. Using Theorems \ref{thm:punctures} and \ref{main2}, we will obtain a complete answer to this question, namely:

\begin{corollary}\label{cor:punctures}
There exists an isometric embedding $\DZ^r \to \DP(S)$ if and only if $r \le [\frac{3g+p-2}{2}]$.
\end{corollary}

Our second application concerns the {\em finite geodesicity} of different graphs of pants decompositions. Combining Theorem \ref{thm:punctures} with the main results of \cite{APS1, APS2}, we will obtain the following:

\begin{corollary}
Let $S$ be the six-times punctured sphere, and let $P,Q$ be two vertices of  $\CP(S)$. Then there are only finitely many geodesics in  $\CP(S)$ between $P$ and $Q$.
\label{geodesics}
\end{corollary}

The analogue of Corollary \ref{geodesics} for $\DP(S)$ is not true. Indeed, we will observe that  if $S$ is a sphere with at least six punctures, one can find pairs of vertices of $\DP(S)$ such that there are infinitely many geodesics in $\DP(S)$ between them. However, we will prove:

\begin{corollary} Given any surface $S$, and any $k$, there exist points in $\DP(S)$ at distance $k$ with only a finite number of geodesics between them.
\label{cor:far}
\end{corollary}

\noindent {\bf Acknowledgements.} The authors would like to thank Jeff Brock and Saul Schleimer for conversations.

\section{Preliminaries}
\label{prelim}

Let $S$ be a connected orientable surface with empty boundary and negative Euler characteristic. The complexity
of $S$ is the number $\kappa(S) =3g-3+p$, where $g$ and $p$ denote, respectively, the genus and number of punctures of $S$.

\subsection{Curves and multicurves}  By a curve on $S$ we will mean
a homotopy class of simple closed curves on $S$; we will often blur the distinction between a curve and any of its representatives. 
We say that a curve $\alpha \subset S$ is essential if no representative of $\alpha$ bounds a disk with at most one puncture. 
The geometric intersection number between two curves $\alpha$ and $\beta$ is defined as \[ i(\alpha, \beta) = \min\{|a \cap b| : a \in \alpha, b \in \beta\}.\]
If $\alpha\neq\beta$ and $i(\alpha, \beta) = 0$ we say that $\alpha$ and $\beta$ are disjoint; otherwise, we say they intersect. We say that the curves $\alpha$ and $\beta$ fill $S$ if $i(\alpha,\gamma)+i(\beta,\gamma)>0$ for all curves $\gamma \subset S$. 

A multicurve is a collection of  pairwise distinct and pairwise disjoint essential curves. The deficiency of a multicurve $C$ is defined as $\kappa(S) - |C|$, where
$|C|$ denotes the number of elements of $C$. 

\subsection{Graphs of pants decompositions}
A pants decomposition $P$ of $S$ is a multicurve that is maximal with respect to inclusion. As such, $P$ consists of exactly $\kappa(S)$ curves, and has a representative $P'$ such that every connected component of $S \setminus P'$ is homeomorphic to a sphere with three punctures, or
pair of pants.  

We say that two pants decompositions of $S$ are related by an elementary move
if they have exactly $\kappa(S) -1$ curves in common, and the remaining two curves either fill a one-holed torus and intersect
exactly once, or else they fill a four-holed sphere and intersect exactly twice. Observe that every elementary move determines a unique subsurface of $S$ of complexity 1; we will say that two elementary moves are disjoint if the subsurfaces they determine are disjoint.

As mentioned in the introduction, the {\em pants graph} $\CP(S)$ of $S$ is the simplicial graph whose vertex set is the set of pants decompositions of $S$, considered
up to homotopy, and where two pants decompositions are adjacent in $\CP(S)$ if and only if they are related by an  elementary move. We turn  $\CP(S)$ into a geodesic metric space by declaring the length of each edge to be 1.  

 The  diagonal pants graph $\DP(S)$ is the simplicial graph obtained from $\CP(S)$ by adding an edge of unit length between any two vertices that differ  by $k\geq 2$ disjoint elementary moves. 

The cubical pants graph $\CCP(S)$
is obtained from $\CP(S)$ by adding an edge {\em of length} $\sqrt{k}$ between any two edges that differ
by $k\ge 2$ disjoint elementary moves.

\section{The forgetful maps: proofs of Theorems \ref{thm:punctures} and \ref{main2}}

The idea of the proof of Theorems \ref{thm:punctures} and \ref{main2} is to use the so-called {\em forgetful maps} to define a distance non-increasing 
projection from the diagonal pants graph to each of its strata. 

\subsection{Forgetful maps} Let $S_1$ and  $S_2$ be connected orientable surfaces with empty boundary.
Suppose that $S_1$ and $S_2$ have equal genus, and that $S_2$ has at most as many punctures as $S_1$.
Choosing an identification between the set of punctures of $S_2$ and a subset of the set of punctures of $S_1$ yields a map $\tilde{\psi}:S_1 \to \tilde{S}_1$, where $\tilde{S}_1$ is obtained by forgetting all punctures of $S_1$ that do not correspond to a puncture of $S_2$. Now, $\tilde{S}_1$ and $S_2$ are homeomorphic; by choosing a homeomorphism $\varphi:\tilde{S}_1\to S_2$, we obtain a map $\psi=\varphi\circ \tilde{\psi}: S_1 \to S_2$. We refer to all such maps $\psi$ as {\em forgetful maps}.

Observe that if $\alpha\subset S_1$ is a curve then $\psi(\alpha)$ is a (possibly homotopically trivial) curve on $S_2$. Also, $i(\psi(\alpha),\psi(\beta))\le i(\alpha,\beta)$ for al curves $\alpha,
\beta \subset S$. As a consequence, if $C$ is a multicurve on $S_1$ then $\psi(C)$ is a multicurve on $S_2$; in particular $\psi$ maps pants decompositions to pants decompositions. Lastly, observe that if $P$ and $Q$ are pants decompositions of $S_1$ that differ by at most $k$ pairwise disjoint elementary moves, then $\psi(P)$ and $\psi(Q)$ differ by at most $k$ pairwise disjoint elementary moves. We summarize these observations as a lemma.

\begin{lemma} Let $\psi:S_1 \to S_2$ be a forgetful map. Then:
\begin{enumerate}
\item If $P$ is a pants decomposition of $S_1$, then $\psi(P)$ is a pants decomposition of $S_2$. 
		
\item If $P$ and $Q$ are related by $k$ disjoint elementary moves in $S_1$, then $\psi(S_1)$ and $\psi(S_2)$ are related by at most $k$ disjoint elementary moves in $S_2$. \hfill$\square$
\end{enumerate}
\label{lem:forget}
\end{lemma}

In light of these properties, we obtain that forgetful  maps induce distance non-increasing maps $$\CP(S_1) \to \CP(S_2)$$ and, similarly, $$\DP(S_1) \to \DP(S_2) \hspace{.5cm}  \text{ and } \hspace{.5cm} \CCP(S_1)\to \CCP(S_2).$$
 
\subsection{Projecting to strata} \label{sec:projection} Let $S$ be a sphere with punctures,  $C$ a multicurve on $S$, and consider $S \setminus C = X_1 \sqcup \ldots \sqcup X_n$. 
For each $i$, we proceed as follows. On each connected component of $S\setminus X_i$ we choose a puncture of $S$. Let $\bar X_i$ be the surface
obtained from $S$ by forgetting all punctures of $S$ in $S\setminus X_i$ except for these chosen punctures. Noting that $\bar X_i$ is naturally homeomorphic to $X_i$, as above 
we obtain a map $$\phi_i: \CP(S) \to \CP(X_i),$$ for all $i = 1, \ldots, n$. 
We now define a map $$\phi_C: \CP(S) \to \CP_C(S)$$  by setting $$ \phi_C(P)= \phi_1(P)\cup \ldots \cup \phi_n(P) \cup C.$$ Abusing notation, observe that we also obtain
maps  $$\phi_C: \DP(S) \to \DP_C(S) \hspace{.5cm}  \text{ and } \hspace{.5cm} \phi_C: \CCP(S) \to \CCP_C(S).$$ The following is an immediate consequence of the definitions and Lemma \ref{lem:forget}:


\begin{lemma}
Let $C\subset S$ be a multicurve. Then
\begin{enumerate}

		\item If $P$ is a pants decomposition of $S$ that contains $C$, then $\phi_C(P) =P$.
		\item For all $P,Q \in \DP(S)$,  $d_{\DP(S)}(\phi_C(P), \phi_C(Q)) \le d_{\DP(S)}(P,Q).$ \hfill$\square$

	\end{enumerate}
\label{project}
\end{lemma}

\subsection{Proof of Theorem \ref{thm:punctures}} After the above discussion, we are ready to give a proof of our first result:

\begin{proof}[Proof of Theorem \ref{thm:punctures}] Let $P$ and $Q$ be vertices of $\DP_C(S)$, and let $\omega$ be a path in $\DP(S)$ between them. By Lemma \ref{project}, $\phi_C(\omega)$ is a path in $\DP_C(S)$ between $\phi_C(P)=P$ and $\phi_C(Q)=Q$, of length at most that of $\omega$. Hence
$\DP_C(S)$ is convex. 
\end{proof}

\begin{remark}[Strata that are not totally geodesic]
{\rm If $S$ has at least 6 punctures, there exist multicurves $C$ in $S$ for which $\DP_C(S)$ fails to be totally geodesic. For instance, let $C= \alpha \cup \beta$, where $\alpha$ cuts off
a four-holed sphere $Z\subset S$ and $\beta \subset Z$. Let $P, Q\in \DP_C(S)$ at distance at least 3, and choose a geodesic path $$P=P_1, P_2, \ldots, P_{n-1}, P_n = Q$$ in $\DP_C(S)$ between them. For each $i= 2, \ldots, n-1$, let $P_i'$ be a pants decomposition obtained from $P_i$ by replacing the curve $\beta$ with a curve $\beta' \subset Z$ such that $P_i$ and $P'_i$ are related by an elementary move for all $i$. Observing that $P$ and $P'_1$ (resp. $P'_{n-1}$ and $Q$) are related by two disjoint elementary moves, and therefore are adjacent in $\DP(S)$, we obtain a geodesic path $$P=P_1, P'_2, \ldots, P'_{n-1}, P_n = Q$$ between $P$ and $Q$ that lies entirely outside $\DP_C(S)$ except at the endpoints. In particular, $\DP(S)$ is not totally geodesic. }
\label{rmk:outside}
\end{remark}

\subsection{Proof of Theorem \ref{main2}}  Let $Y$ be a subsurface of $S$ such that $Y$ has the same genus as $S$. Each component of $\partial Y$ bounds in $S$ a punctured disc. As in Section \ref{sec:projection}, we may define a projection $$\phi_Y: \CP(S) \to \CP(Y),$$  this time by choosing one puncture of $S$ in each punctured disc bounded by a component of $\partial Y$. Let $C$ be the union of a pants decomposition of $S\setminus Y$ with all the boundary components of $Y$. We obtain a map: 
$$\phi_C: \CP(S) \to \CP_C(S)$$ by setting $\phi_C(P)= \phi_Y(P) \cup C.$ Again abusing notation, we also have maps  $$\phi_C: \DP(S) \to \DP_C(S) \hspace{.5cm}  \text{ and } \hspace{.5cm} \phi_C: \CCP(S) \to \CCP_C(S).$$

\begin{proof}[Proof of Theorem \ref{main2}]
We prove part (2) only, since the proof of part (1) is completely analogous to that of Theorem \ref{thm:punctures}. Since $Y$ is connected, Lemma \ref{lem:forget} implies that, for any pants decompositions $P$ and $Q$ of $S$,  the projections $\phi_C(P)$ and $\phi_C(Q)$ differ by at most as many elementary moves as $P$ and $Q$ do.  
In other words, the maps $\phi_C: \CP(S) \to \CP_C(S)$ and $\phi_C: \CCP(S) \to \CCP_C(S)$ are distance non-increasing. 

Suppose, for contradiction, there exist $P,Q \in \CP_C(S)$ and a geodesic $\omega: P=P_0,P_1, \ldots, P_n=Q$ such that $P_i\in \CP_C(S)$ if and only if $i=0,n$. 
Let $\tilde{P}_i = \phi_C(P_i)$ for all $i$, noting that $\tilde{P}_0=P$ and  $\tilde{P}_n=Q$. Moreover, since $P$ contains $C$ but $P_1$ does not, then
$\tilde{P}_1= \phi_C(P_1) = P_0$. Therefore,  the projected path $\phi_C(\omega)$ is strictly shorter (both in $\CP(S)$ and $\CCP(S))$ than $\omega$, and the result
follows. \end{proof}

\begin{remark}
{\rm If $S$ has at least $3$ punctures (and genus at least $1$), then there are multicurves $C\subset S$ such that $\DP_C(S)$ is not totally geodesic. One such example is $C=\alpha\cup\beta$, where $\alpha$ cuts off a four-holed sphere $Z\subset S$, and $\beta \subset Z$; compare with Remark \ref{rmk:outside}.}
\label{rmk:outside2}
\end{remark}

\section{Convex Farey graphs: proof of Theorem \ref{thm:general}}
The main goal of this section is to prove that, for any surface $S$ and any multicurve $C\subset S$ of deficiency 1, the stratum $\DP_C(S)$ is convex in $\DP(S)$. This is the equivalent theorem for $\DP(S)$ of the main result of \cite{APS1}, which states that any isomorphic copy of a Farey graph in $\CP(S)$ is totally geodesic; compare with Remark \ref{rmk:outside3} below.

\subsection{Geometric subsurface projections}
We recall the definition and some properties of subsurface projections, as introduced by Masur-Minsky \cite{MM2}, in the particular case where the subsurface $F$ has complexity $1$.



Let $C\subset S$ be a deficiency 1 multicurve; as such, $S\setminus C$ contains an incompressible  subsurface $F$ of complexity 1, thus either a one-holed torus or a four-holed sphere. 
Let $\alpha\subset  S$ be a curve that intersects $F$ essentially. Let $a$ be a connected component of $\alpha \cap F$; as such, $a$ is either a curve in $F$, or an arc with endpoints on $\partial F$. Now, there is exactly one curve $\gamma_a$ in $F$ that is disjoint from $a$, and we refer to the pants decomposition $\gamma_a \cup C$ as a  projection of $\alpha$. 

We write $\pi_C(\alpha)$ to denote the set of all projections of $\alpha$, each counted once; we note that $\pi_C(\alpha)$  depends only on $\alpha$ and $C$. If $D\subset S$ is a multicurve, then $\pi_C(D)$ is the union
of the projections of the elements of $D$. In the special case where a multicurve $D$ does not intersect $F$ essentially, we set $\pi_C(D) = \emptyset$. Finally, observe that if $D$ contains a curve $\beta$ that is contained in $F$, then $\pi_C(D)=\beta\cup C$.\\

We will need the following notation. Given an arc $a\subset F$ with endpoints in $\partial F$, we call it a {\it wave} if its endpoints belong to the same boundary component of $F$; otherwise, we call it a {\it seam}. These two types of arcs are illustrated in Figure \ref{fig:arcs}. Observe that any arc in a one-holed torus is a seam. 

\begin{figure}[h]
\leavevmode \SetLabels
\endSetLabels
\begin{center}
\AffixLabels{\centerline{\epsfig{file =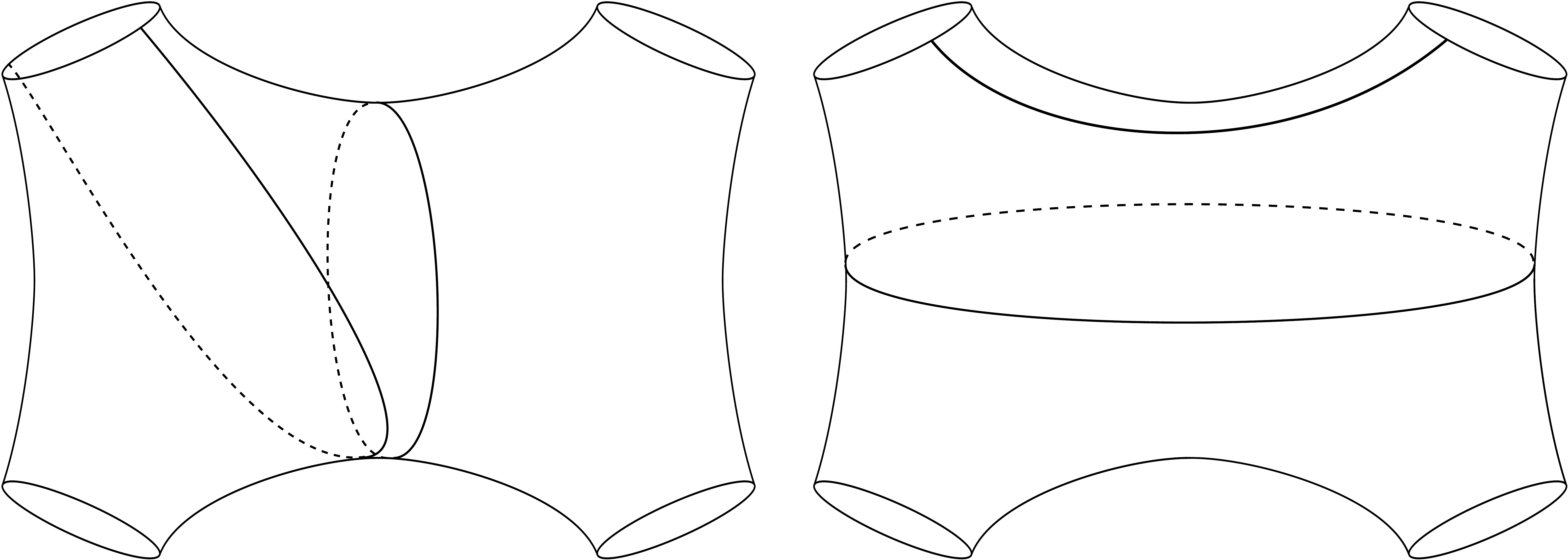,height=4.5cm,angle=0}\hspace{.2cm} \epsfig{file =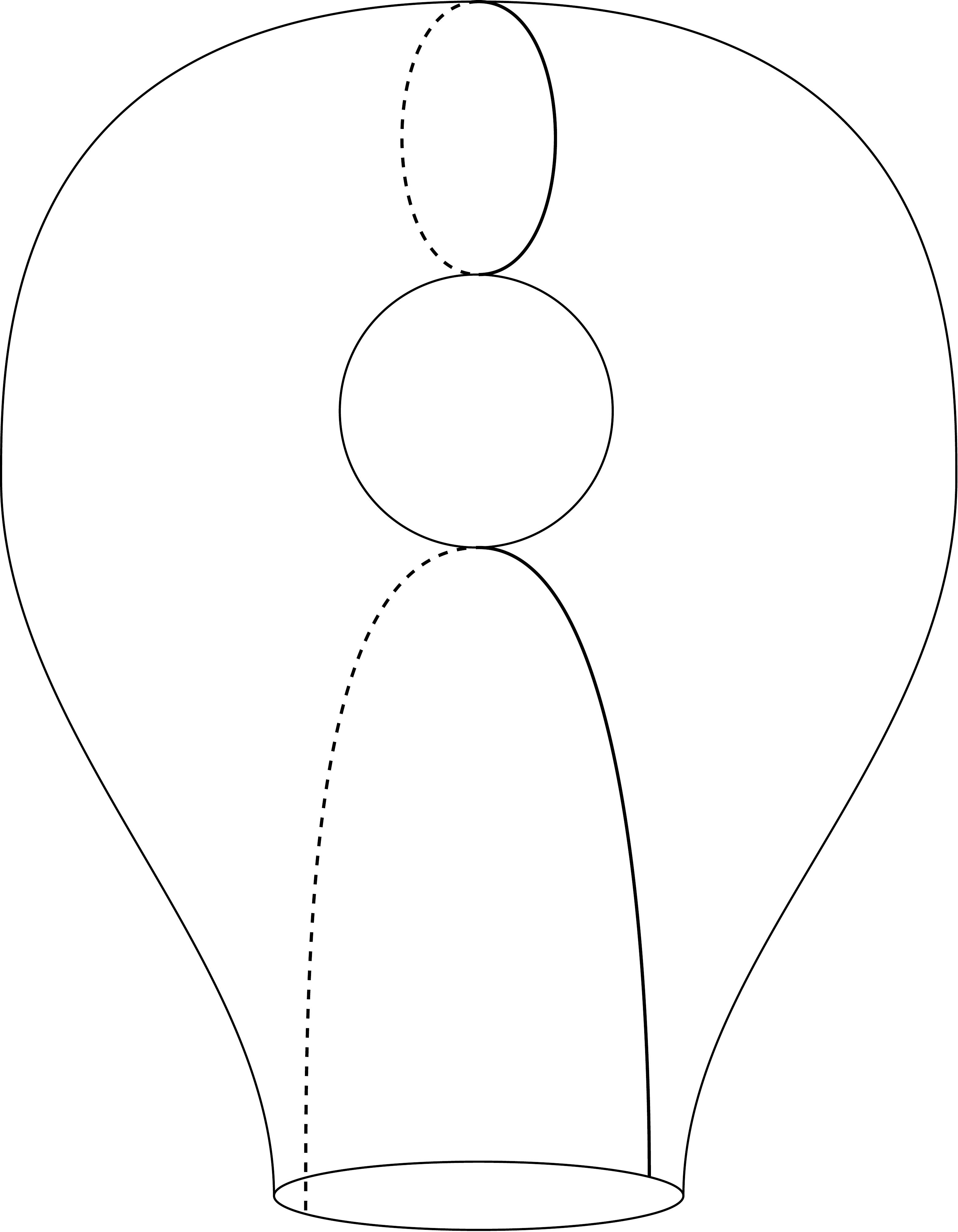,height=4.5cm,angle=0}}}
\end{center}
\caption{A wave and a seam on a four holed sphere, a wave on a torus (and their projections)} \label{fig:arcs}
\end{figure}

\subsection{Preliminary lemmas}
Unless specified, $C$ will always be a deficiency $1$ multicurve on $S$, and $F$ the unique (up to isotopy) incompressible subsurface  of complexity 1 in $S\setminus C$.
Observe that, in this case, $\DP_C(S)$ is isomorphic to $\CP(F)$. We will need some lemmas, similar to those previously used in \cite{APS1}. First, since any two disjoint arcs
in a one-holed torus project to curves that intersect at most once, we have:

\begin{lemma}\label{lem:disjointT} 
Suppose $F$ is a one-holed torus, and let $\alpha,\beta\subset S$ be curves with $i(\alpha,\beta)=0$. Then, $\pi_C(\alpha\cup \beta)$ has diameter at most 1 in $\CP(F)$.\hfill $\square$
\end{lemma}



For $F$ a four-holed sphere, we need the following instead:

\begin{lemma}\label{lem:dis2} 
Let $F$ be a four-holed sphere, and let $\alpha,\beta\subset F$ be curves with $i(\alpha,\beta)\le 8$. Then , $d_{\CP(F)}(\alpha, \beta) \le 2$. 
\end{lemma}

\begin{proof}
It is well-known that curves in $F$ correspond to elements of $\BQ \cup \{\infty\}$ expressed in lowest terms. Moreover, if $\alpha=p/q$ and $\beta=r/s$, then $i(\alpha,\beta) = 2|ps-rq|$. Up to the action of ${\rm SL}(2, \BZ)$, we may assume that $\alpha = 0/1$ and $\beta = r/s$. Since $i(\alpha, \beta) \le 8$, then $r\le 4$; suppose, for concreteness, that $r=4$. We are looking for a curve $\gamma=u/v$ with $i(\alpha,\gamma)= i(\beta,\gamma)=2$. Since $i(\alpha, \gamma) =2$, then $u=1$. Since ${\rm gcd}(r,s) =1$, then $s=4k+1$ or $s=4k+3$, for some $k$. In the former case, we set $v = k$ and, in the latter case, we set $v=k+1$. \end{proof}

As a direct consequence of the above, we obtain:.

\begin{lemma}\label{lem:disjoint} 
Suppose $F$ is a four-holed sphere, and let $\alpha,\beta\subset S$ be curves with $i(\alpha,\beta)=0$. Then, $\pi_C(\alpha\cup \beta)$ has diameter at most 2 in $\CP(F)$. \hfill $\square$
\end{lemma}

We will also need the following refinement of Lemma 9 of \cite{APS1}:

\begin{lemma}\label{lem:twocurves}
Let $P$ be a pants decomposition of $S$ with $\partial F \not\subset P$ and let $\alpha_1 \in P$ be a curve that essentially intersects $F$. Then there exists $\alpha_2\in P\setminus \alpha_1$ that essentially intersects $F$. Futhermore $\alpha_2$ can be chosen so that $\alpha_1$ and $\alpha_2$ share a pair of pants.
\end{lemma}

\begin{proof}
We prove the first assertion.  Suppose, for contradiction, that $\alpha_1$ is the only element of $P$ that essentially intersects $F$. Since $\partial F \not\subset P$ and $P\setminus \alpha_1$ has deficiency 1, then $P\setminus \alpha_1 \cup \partial F$ is a pants decomposition of $S$, which is impossible since its complement contains $F$, which has complexity 1. Thus there is a curve $\tilde{\alpha}\neq \alpha_1$ of $P$ essentially intersecting $F$.

To show that a second curve $\alpha_2$ can be chosen sharing a pair of pants with $\alpha_1$, consider any path $c:[0,1]\longrightarrow F$ with $c(0)\in \alpha_1$ and $c(1)\in \tilde{\alpha}$ and choose $\alpha_2$ to be the curve on which $c(\tau)$ lies, where $\tau = \min\{t | c(t) \in P\setminus \alpha_1 \}$. Thus there is a path between $\alpha_1$ and $\alpha_2$ which intersects $P$ only at its endpoints. Therefore, $\alpha_1$ and $\alpha_2$ share a pair of pants. 
\end{proof}

The next lemma will  be crucial to the proof of the main result.  In what follows, $d_C$ denotes distance in $\DP_C(S)=\CP_C(S)$

\begin{lemma}\label{lem:mainlemma}
Let $F$ be a four-holed sphere. Let $P,P'$ be pants decompositions at distance $1$  in $\DP(S)$ and let $\alpha\in P$, $\alpha' \in P'$ be curves related by an elementary move. Suppose both curves intersect $F$ essentially, and let $a$ be an arc of $\alpha$ on $F$. Then at least one of the following two statements holds:
\begin{enumerate}[i)]\item\label{casei}There exists $\tilde{\alpha} \in P\cap P'$ and an arc $\tilde{a}$ of $\tilde{\alpha}$ on $F$ such that 
$$d_C(\pi_C(a), \pi_C(\tilde{a})) \leq 1.$$
\item\label{caseii} The curve $\alpha'$ satisfies
$$d_C(\pi_C(a), \pi_C(\alpha')) \leq 2.$$
\end{enumerate}
\end{lemma}
\begin{remark} The example illustrated in Figure \ref{fig:lemproof1-2} shows that case \ref{caseii}) does not cover everything. The two arcs labelled $a$ and $a'$, which can indeed be subarcs of curves related by an elementary move, produce projected curves that are distance $3$.
\end{remark}
\begin{proof}[Proof of lemma \ref{lem:mainlemma}]
First, we choose representatives of $\alpha$ and $\alpha'$ that realize $i(\alpha,\alpha')$; abusing notation, we denote these representatives by $\alpha$ and $\alpha'$. 

Let $a$ (resp. $a'$) be any arc of $\alpha$ (resp. $\alpha'$) in $F$. First, observe that if $i(a,a')\le 1$ then
$i(\pi_C(a), \pi_C(a')) \leq 8$. The same is true if at least one of $a$ or $a'$ is a wave. In either case, Lemma \ref{lem:dis2} yields that $d_C(\pi_C(a), \pi_C(\alpha')) \leq 2$ and thus the result follows. 

Therefore, it suffices to treat the case where $i(a,a')=2$ and both $a,a'$ are seams. In this case, we will show that  conclusion \ref{casei}) holds.

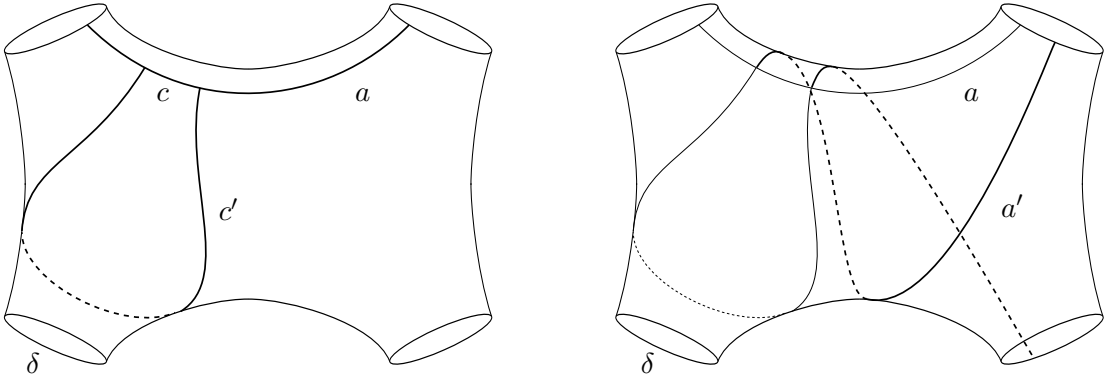
\begin{figure}[h]
\leavevmode \SetLabels
\L(.23*.4) $c'$\\
\L(.34*.73) $a$\\
\L(.83*.73) $a$\\
\L(.86*.41) $a'$\\
\L(.18*.73) $c$\\
\L(.075*-.023) $\delta$\\
\L(.57*-.023) $\delta$\\
\endSetLabels
\begin{center}
\AffixLabels{\centerline{\epsfig{file =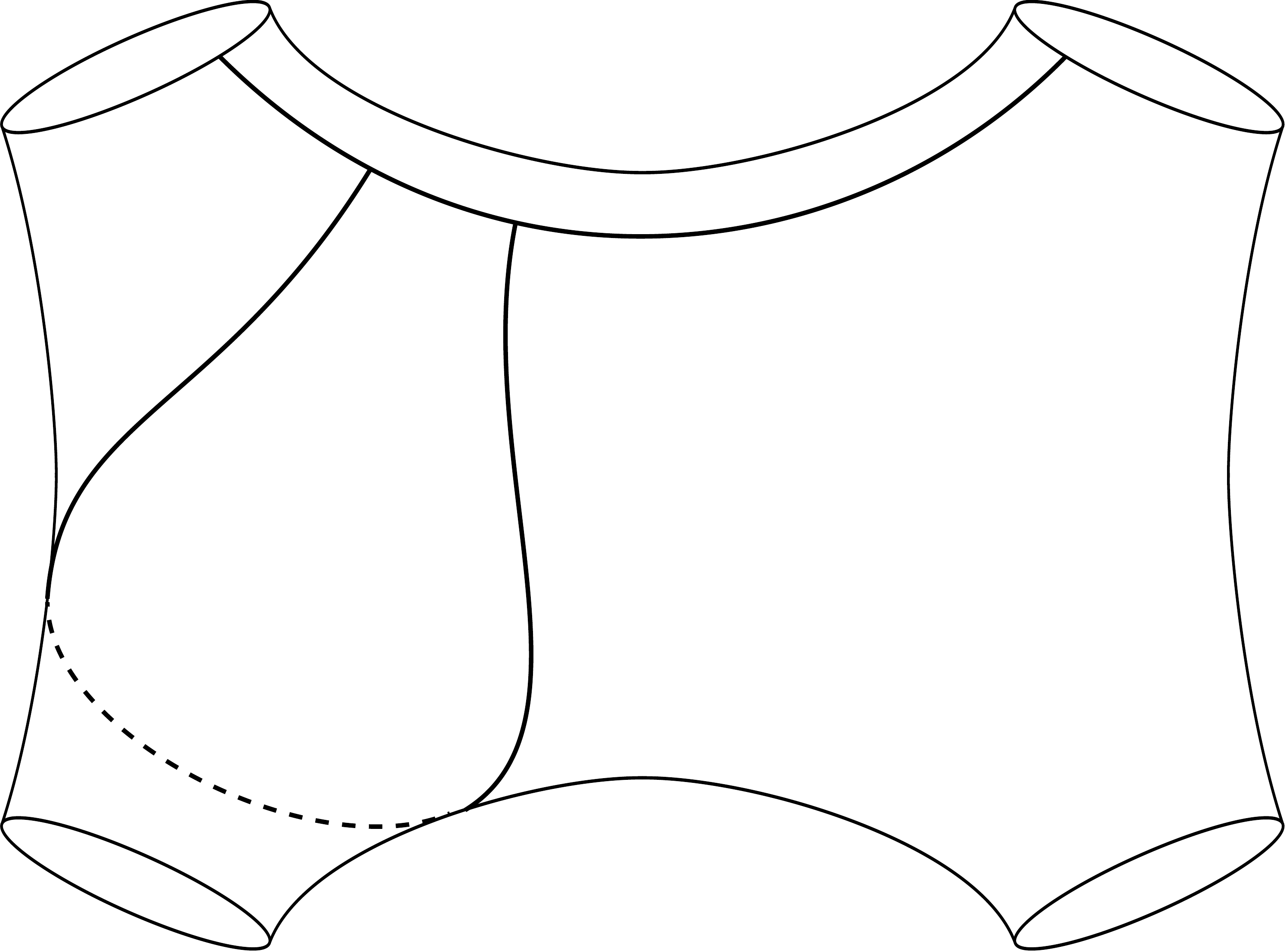,width=6.5cm,angle=0}\hspace{1.5cm} \epsfig{file =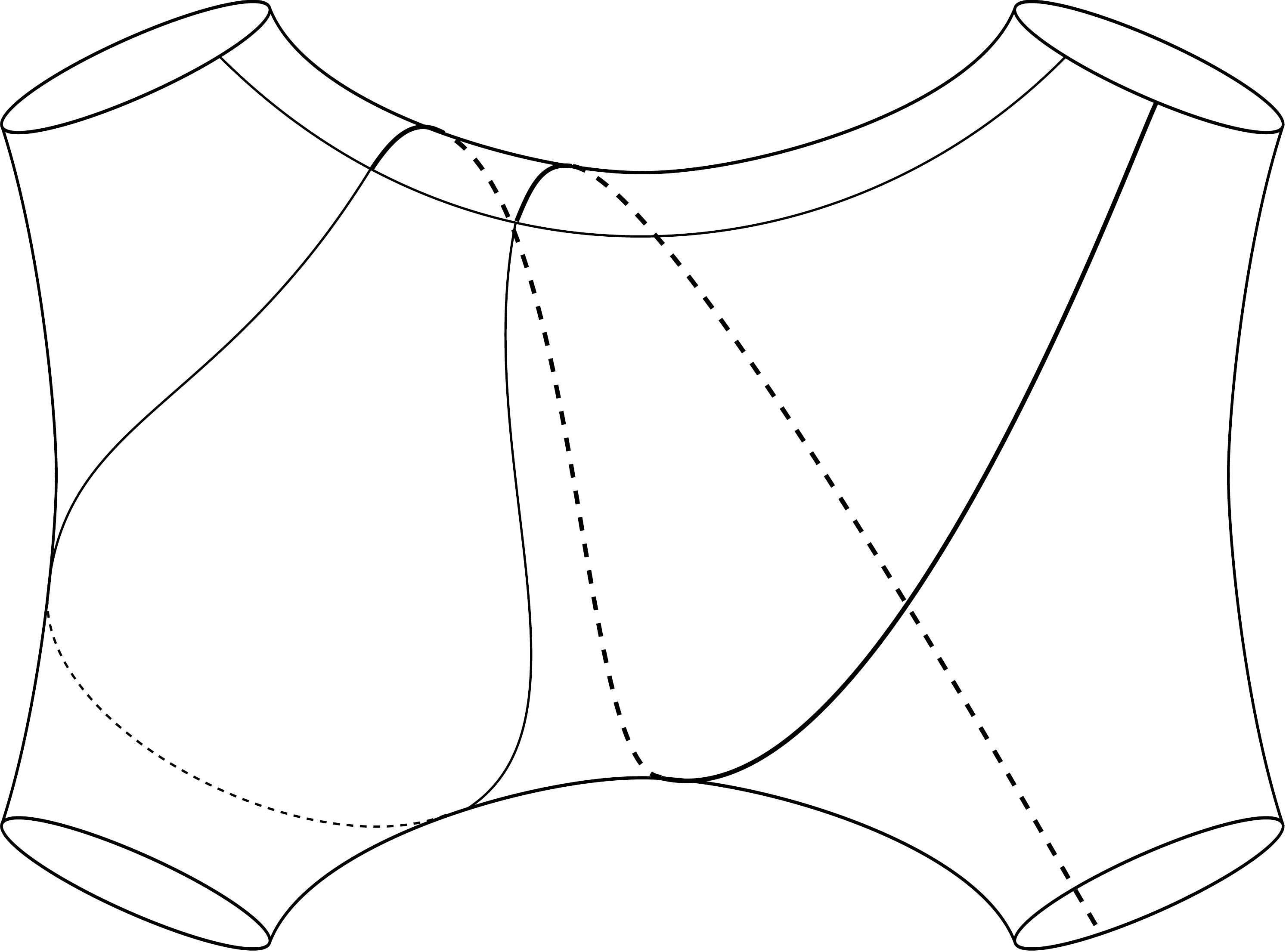,width=6.5cm,angle=0}}}
\end{center}
\caption{The arcs $a$ and $a'$, shown in the four-holed sphere $F$.} \label{fig:lemproof1-2}
\end{figure}

Consider the subarcs $c\subset a$ and $c'\subset a'$ between the two intersection points of $a$ and $a'$. As $\alpha$ and $\alpha'$ intersect twice and are related by an elementary move, they fill a unique (up to isotopy) incompressible four-holed sphere $X\subset S\setminus(P\cap P')$; in particular, they have algebraic intersection number 0. 
As the endpoints of $a'$ must lie on different boundary components of $F$, by checking the different possibilities, the only topological configuration of the two arcs is the one shown in Figure \ref{fig:lemproof1-2}.

Denote $\delta$ the curve in the homotopy class of $c \cup c'$. This curve is a boundary curve of $F$, as is illustrated in Figure  \ref{fig:lemproof1-2}, but is also one of the boundary curves of $X$. This is because $\delta$ is disjoint from all the boundary curves of $X$, and is not homotopically trivial since $i(\alpha,\alpha')=2$. In $X$, 
the curve $\alpha'$ and the arc $c$ determine two boundary curves of $X$, namely $\delta$ and another one we shall denote $\tilde{\delta}$ (see Figure \ref{fig:lemproof3}).

\begin{figure}[h]
\leavevmode \SetLabels
\L(.44*.46) $\alpha$\\
\L(.53*.73) $\alpha'$\\
\L(.6*.38) $c$\\
\L(.67*.96) $\delta$\\
\L(.67*-.025) $\tilde{\delta}$\\
\endSetLabels
\begin{center}
\AffixLabels{\centerline{\epsfig{file =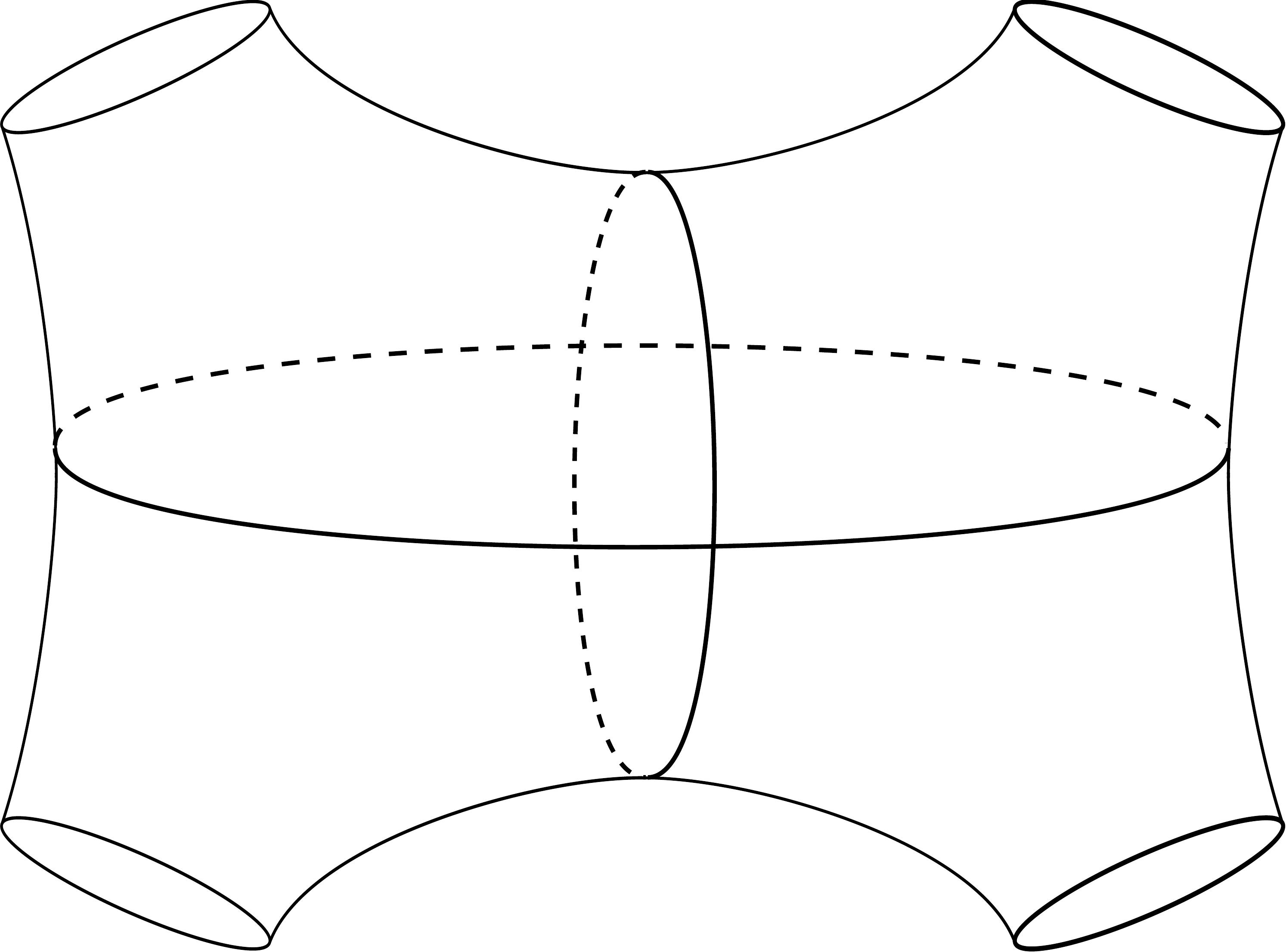,width=6.5cm,angle=0}}}
\end{center}
\caption{The surface $X$ determined by $\alpha$ and $\alpha'$} \label{fig:lemproof3}
\end{figure}

Note that the curve $\tilde{\delta}$ by definition shares a pair of pants with $\alpha$ in $P$ and with $\alpha'$ in $P'$.
Now, in $F$ the arc $\tilde{c}$, illustrated in Figure \ref{fig:lemproof4}, is an essential subarc of $\tilde{\delta}$. Now clearly
$$
d_C(\pi_C(\tilde{c}),\pi_C(a))=1;
$$
see again Figure \ref{fig:lemproof4}. As $\tilde{\delta} \in P\cap P'$, we can conclude that \ref{casei}) holds and thus the lemma is proved.
\end{proof}

\begin{figure}[h]
\leavevmode \SetLabels
\L(.41*.46) $\tilde{c}$\\
\L(.855*.7) $\tilde{c}$\\
\L(.64*.72) $a$\\
\L(.6*.35) $\pi(a)$\\
\L(.74*.47) $\pi(\tilde{c})$\\
\endSetLabels
\begin{center}
\AffixLabels{\centerline{\epsfig{file =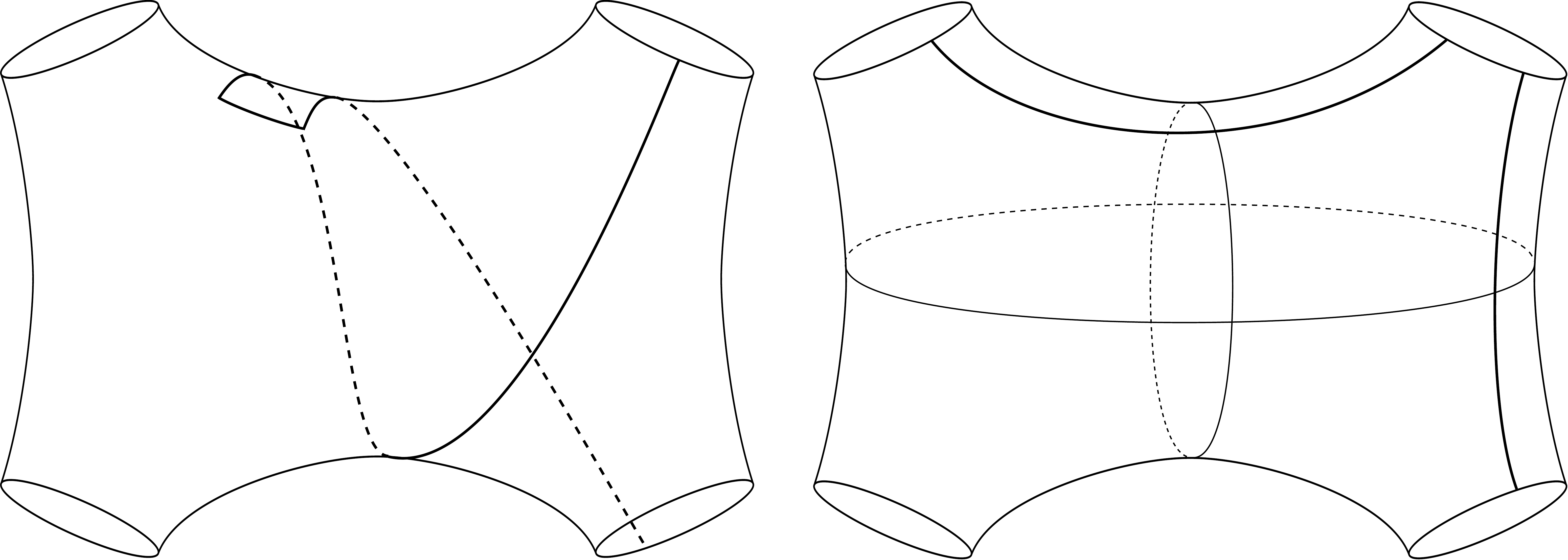,width=13.5cm,angle=0}}}
\end{center}
\caption{The surface $X$ determined by $\alpha$ and $\alpha'$} \label{fig:lemproof4}
\end{figure}

\subsection{Proof of Theorem \ref{thm:general}} 
We will prove the following more precise statement, which implies part (a) of Theorem \ref{thm:general}. Again, $d_C$ denotes distance in $\DP_C(S)$ (or, equivalently, in $\CCP_C(S)$, since both are isomorphic).

\begin{theorem}
Let $C\subset S$ be a deficiency 1 multicurve, and let $F$ be the unique (up to isotopy) incompressible surface of complexity 1 in $S\setminus C$. Let $P,Q \in \DP_C(S)$, and let $P=P_0,\hdots,P_n=Q$ be a path in $\DP(S)$ between them. If  $\partial F \not\subset P_k$ for at least one $k$,
then $n>d_C(P,Q)$.
\label{thm:themeat}
\end{theorem}

\begin{remark}
In fact, Theorem \ref{thm:themeat} also implies part (b) of Theorem \ref{thm:general}, since $d_{\CCP(S)}(P,Q) \ge d_{\DP}(S)(P,Q)$. 
\end{remark}

\begin{proof}[Proof of Theorem \ref{thm:themeat}]
Let $P=P_0,\hdots,P_n=Q$ be a path in  $\DP(S)$. It suffices to consider the case that $\partial F$ is not contained in $P_k$ for $k\notin \{0,n\}$. 
The proof relies on projecting to curves in $\pi_C(P_k)$ to produce a path  $\tilde{P}_0, \ldots, \tilde{P}_n$ in $\DP_C(S)$, in such a way that  $d(\tilde{P}_i,\tilde{P}_{i+1})\leq 1$. The strict inequality will come from the observation that, since $P_0$ contains $\partial F$ but $P_1$ does not, then $\pi_C(P_1)= P_0$.

We set $\tilde{P}_0:= P_0$ and $\tilde{P}_1:= \pi_C(P_1)=P_0$. Proceeding inductively, suppose that we have constructed $\tilde{P}_k$ by projecting the arc $a_k$ of the
curve $\alpha_1^k \in P_k$. Let $\alpha_2^k$ the curve of $P_k$ whose existence is guaranteed by Lemma \ref{lem:twocurves}. We begin with the easier case when $F$ is a one holed torus.\\


\noindent{\it \underline{The case where $F$ is a one holed torus}}\\

\noindent We proceed as follows for $k\leq n-1$.
At step $k+1$, at least $\alpha_1^k$ or $\alpha_2^k$ belong to $P_{k+1}$. We set $\tilde{P}_{k+1}$ to be the projection of any arc of such curve. By Lemma \ref{lem:disjointT} , we have $d_C(\tilde{P}_{k},\tilde{P}_{k+1}) \leq 1$.  Since $Q$ contains $\partial F$ but $P_{n-1}$ does not, then $\pi_C(P_{n-1})=Q$ and therefore $\tilde{P}_{n-1} = Q$. In this way we obtain the desired path, which has length at most $n-2$. \\

\noindent{\it  \underline{The case where $F$ is a four holed sphere}}\\

\noindent We proceed as follows for $k\leq n-2$. There are several cases to consider:

\begin{enumerate}[1.]
\item[{\bf 1.}] If $\alpha_1^k \in P_{k+1}$, then we define $\tilde{P}_{k+1}:=\tilde{P}_{k}$.
\end{enumerate}

 \noindent If $\alpha_1^k \notin P_{k+1}$, denote by $\alpha_1^{k+1}$ the curve in $P_{k+1}$ that intersects
 $\alpha_1^k $ essentially.

\begin{enumerate}[2.]
\item[{\bf 2.}] Suppose that $\alpha_1^{k+1}$ intersects $F$ essentially. As $\alpha_1^{k+1}$ and $\alpha_1^{k}$ are related by an elementary move, and since $\alpha_1^{k}$ shares a pair of pants with $\alpha_2^{k}$, it follows that $\alpha_2^{k}$ also belongs to $P_{k+1}$ and that it shares a pair of pants with $\alpha_1^{k+1}$.



\begin{enumerate}[i.]\item If $d(\pi_C(a_k), \pi_C(\alpha_1^{k+1}))>2$, then by Lemma \ref{lem:mainlemma}, there exists a curve $\tilde{\alpha}$ in $P_k\cap P_{k+1}$ (not necessarily $\alpha_2^{k}$) with an arc $\tilde{a}$ such that $d(\pi_C(a_k), \pi_C(\tilde{a}))=1$. In this case we set
$
\tilde{P}_{k+1}:= C \cup \pi_C(\tilde{a})
$, noting that $d(\tilde{P}_{k}, \tilde{P}_{k+1})=1$.

\item Otherwise we have $d(\pi_C(a_k), \pi_C(\alpha_1^{k+1}))\leq 2$. Now, at least one of $\alpha_1^{k+1}$ or $\alpha_2^{k}$ belongs to $P_{k+2}$. We define $\tilde{P}_{k+2}$ to be the projection of any arc of $\{ \alpha_1^{k+1}, \alpha_2^{k}\} \cap P_{k+2}$, observing that $d(\tilde{P}_{k}, \tilde{P}_{k+2})\le 2$, by the above observations and/or Lemma \ref{lem:disjoint}. (In the case that $d(\tilde{P}_{k}, \tilde{P}_{k+2})= 2$, we choose any $\tilde{P}_{k+1}$ that contains $\partial F$ and is adjacent to both $\tilde{P}_{k}$ and  $\tilde{P}_{k+2}$.)
\end{enumerate}

\end{enumerate}

\begin{enumerate}[1.]
\item[{\bf 3.}]  Finally, suppose that $\alpha_1^{k+1}$ does not essentially intersect $F$. By Lemma \ref{lem:twocurves}, $P_{k+1}$ contains a curve $\tilde{\alpha}$ which essentially intersects $F$ and which shares a pair of pants with $\alpha_2^k$. The curve $\tilde{\alpha}$ may or may not belong to $P_k$ but, in either case, observe that it is disjoint from $\alpha_1^k$. As such, by Lemma \ref{lem:disjoint} we have:
$$
d(\pi_C(a_k),\pi_C(\tilde{\alpha})) \leq 2.
$$
As before, at least $\tilde{\alpha}$ or  $\alpha_2^k$ belongs to $P_{k+1}$ so we choose any arc of $\{ \tilde{\alpha}, \alpha_2^k\} \cap P_{k+2}$ to produce $\tilde{P}_{k+2}$, which is at distance at most $2$ from $\tilde{P}_{k}$.\\
\end{enumerate}

This process will provide us with all $\tilde{P}_k$ up until either $k=n$ or $k=n-1$. We conclude by noticing that if $\tilde{P}_{n-1}$ was obtained as the projection of a curve in $P_{n-1}$ then, as in the initial step, $\tilde{P}_{n-1}=\pi_C(P_{n-1})=P_n$ and we set $\tilde{P_n}:=P_n$.
\end{proof}

Rephrasing Theorem \ref{thm:themeat}, we obtain the following nice property of geodesics in $\DP(S)$:

\begin{corollary}
Let $C\subset S$ be a deficiency 1 multicurve, and let $F$ be the unique (up to isotopy) incompressible surface of complexity 1 in $S\setminus C$. Let $P,Q \in \DP_C(S)$ and let $\omega$ be any geodesic from $P$ to $Q$. Then every vertex of $\omega$ contains $\partial F$.
\end{corollary}

\begin{remark}
{\rm Combining Theorem \ref{thm:general} with a similar argument to that in Remark \ref{rmk:outside} we obtain that, if $\kappa(S) \ge 3$, then $\DP_C(S)$ is not totally geodesic in $\DP(S)$.}
\label{rmk:outside3}
\end{remark}

\section{Consequences}

Let $S$ be a connected orientable surface of negative Euler characteristic and with empty boundary. Let $R=R(S)$ be the number $[\frac{3g+p-2}{2}]$ where $g$ and $p$ are, respectively, the genus and number of punctures of $S$. Note that $R$ is the maximum number of pairwise distinct, pairwise disjoint complexity 1 subsurfaces of $S$. As mentioned in the introduction, by work of Behrstock-Minsky \cite{bemi}, Brock-Farb \cite{brfa}, and Masur-Minsky \cite{MM2}, $R$ is precisely the {\em geometric rank} of the Weil-Petersson completion $\hat{T}(S)$ (and thus also of the pants graph $\CP(S)$, by Brock's result \cite{brock}).

We will say that a multicurve $D \subset S$ is {\em rank-realizing} if $S\setminus D$ contains $R$ pairwise distinct, pairwise disjoint incompressible subsurfaces of $S$, each of
complexity 1. We have:

\begin{proposition}
Let $D\subset S$ be a rank-realizing multicurve. Then $\DP_D(S)$ is convex in $\DP(S)$.
\label{prop:rank}
\end{proposition}

\begin{proof}
Let $X_1,\ldots, X_R$ be the $R$ complexity 1 incompressible subsurfaces in $S\setminus D$. Let $P,Q$ be any two vertices of $\DP_C(S)$, and denote by $\alpha_i$ (resp. $\beta_i$) the unique curve in $P$ (resp. $Q$) that is an essential curve in $X_i$. Consider any path $\omega$ between $P$ and $Q$. As  in the proof
of Theorem \ref{thm:general}, we may project $\omega$ to a path $\omega_i$ in $\DP(X_i)$ from $\alpha_i$ to $\beta_i$, such that ${\rm length}(\omega_i)\le {\rm length}(\omega)$ for all $i$. In fact, by repeating vertices at the end of some of the $\omega_i$ if necessary, we may assume that all the $\omega_i$ have precisely $N+1$ vertices, with $N \le {\rm length}(\omega)$. 

Now, if $\omega_{i,j}$ denotes the $j$-th vertex along $\omega_i$, then $P_j=\omega_{1,j} \cup  \ldots, \omega_{R,j} \cup D$ is a vertex of $\DP_D(S)$. In this way, we obtain
a path $P=P_0, \ldots, P_N=Q$ that is entirely contained in $\DP_D(S)$ and has length $N\le {\rm length}(\omega)$, as desired. 
\end{proof}

As a consequence, we obtain our first promised corollary. Here, $\DZ^r$ denotes the graph obtained from the cubical lattice $\BZ^r$  by adding
the diagonals to every $k$-dimensional cube of $\BZ^r$, for $k\le r$. 


\begin{named}{Corollary \ref{cor:punctures}}
There exists an isometric embedding $\DZ^r \to \DP(S)$ if and only if $r \le R$.
\end{named}

\begin{proof}
We exhibit an isometric embedding of $\DZ^R$, as the construction for $r\le R$ is totally analogous. Let $D$ be a rank-realizing multicurve, and let $X_1,\ldots, X_R$ be the $R$ incompressible subsurfaces of complexity 1 in $S\setminus D$. For $i=1, \ldots, R$, choose
a bi-infinite geodesic $\omega_i$ in $\DP(X_i)$. Let $\omega_{i,j}$ denote the $j$-th vertex along $\omega_i$. Then subgraph of $\DP(S)$ spanned by 
$\{\omega_{i,j}\cup D: 1\le i\le R, j\in \BZ\}$ is isomorphic to $\DZ^R$, and is convex in $\DP(S)$ by Proposition \ref{prop:rank}.

Conversely, every $k$-dimensional cube in $\DZ^r$ singles out a unique (up to isotopy) collection of $k$ pairwise distinct, pairwise disjoint incompressible subsurfaces of $S$, each of complexity 1. Thus the result follows.
\end{proof}

Theorem \ref{thm:general} also implies the following:

\begin{named}{Corollary \ref{cor:far}}
 Given any surface $S$, and any $k$, there exist points in $\DP(S)$ at distance $k$ with only a finite number of geodesics between them.
\end{named}

\begin{proof}
Let $D$ be a rank-realizing multicurve, and let $X_1,\ldots, X_R$ be the $R$ incompressible subsurfaces of complexity 1 in $S\setminus D$. Let $P\in \DP_D(S)$ and let $\alpha_i$ be the unique curve of $P$ that essentially intersects $X_i$. For all $i=1, \ldots, R$, choose a geodesic ray $\omega_i$ in $\CP(X_i)$ issued from $\alpha_i$, and denote by $\omega_{i,j}$ the $j$-th vertex along $\omega_i$. Let $Q_k= \omega_{1,k} \cup  \ldots, \omega_{R,k} \cup D$. By Proposition \ref{prop:rank}, $d_{\DP(S)}(P, Q_k) = k$ for all $k$.

By Theorem \ref{thm:themeat}, any geodesic path in $\DP(S)$ between $P$ and $Q_k$ is entirely contained in $DP_D(S)$. Furthermore, it follows from the choice of $Q_k$ that the projection to $\DP(X_i)$ of a geodesic path in $\DP(S)$ between $P$ and $Q_k$ is a geodesic path in $\DP(X_i)$ between
$\alpha_i$ and $\omega_{i,k}$. The result now follows because $\DP(X_i)$ is isomorphic to a Farey graph, and there are only finitely many geodesic paths between any two points of a Farey graph.
\end{proof}

\section*{Appendix: Strong convexity vs. finite geodesicity}\label{cuicui}

In this section, we prove:



\begin{proposition}
Let $S$ be a connected orientable surface of negative Euler characteristic. We denote by $\CC(S)$ either $\CP(S)$ or $\DP(S)$ or $\CCP(S)$. The following two properties are equivalent:
\begin{enumerate}
\item Given a simple closed curve $\alpha\subset S$, the subset $\CC_\alpha(S)$ of $\CC(S)$ spanned by the vertices corresponding to pants decompositions containing $\alpha$ is totally geodesic,
\item Given two vertices $P,Q\in \CC(S)$, there are finitely many geodesic segments joining $P$ to $Q$ in $\CC(S)$.
\end{enumerate}
\label{prop:cyrilito}
\end{proposition}

\begin{proof}
Let us first show that $1)\Rightarrow 2)$. More precisely, we will show that  $[\text{not }2)]\Rightarrow  [\text{not } 1)]$. Consider an infinite family $\{\omega_n,n\in\BN\}$ of distinct geodesic arcs joining two points $P,Q\in\CC(S)$ and denote the vertices of $\omega_n$ by $\omega_{n,i}$, $0\leq i\leq q(n)$, where $d_{\CC(S)}(\omega_{n,i-1},\omega_{n,i})=1$ when $\CC(S)=\CP(S)$ or $\DP(S)$ and $d_{\CC(S)}(\omega_{n,i-1},\omega_{n,i})= \sqrt{k}$, for some integer  $k \le [\frac{3g+p-2}{2}]$, when $\CC(S)=\CCP(S)$. Notice that $q(n)=d_{\CC(S)}(P,Q)$ if $\CC(S)=\CP(S)$ or $\DP(S)$ but $q(n)\leq d_{\CC(S)}(P,Q)$ may depend on $n$ for $\CC(S)=\CCP(S)$. In the latter case, we first extract a subsequence such that $q(n)$ does not depend on $n$. Let $0<i_0<q(n)=q$ be the smallest index such that $\{\omega_{n,i_0}\}$ contains infinitely many distinct points. Let us extract a subsequence such that $\omega_{n,i}$ does 
not depend on $n$ for $i<i_0$ and $\{\omega_{n,i_0}\}\neq \{\omega_{m,i_0}\}$ for any $n\neq m$. Extract a further subsequence such that the set of leaves of $\omega_{n,i_0-1}$ that are not leaves of $\omega_{n,i_0}$ does not depend on $n$ (namely we fix the subsurfaces in which the elementary moves happen).  Extract a subsequence one last time so that for any $i$, $\omega_{n,i}$ converges in the Hausdorff topology to a geodesic lamination $\omega_{\infty,i}$.\\
\indent
Since all the pants decompositions $\{\omega_{n,i_0}\}$ are distinct, there is a leaf $\alpha$ of $\{\omega_{n,i_0-1}\}$ which is not a leaf of any $\{\omega_{n,i_0}\}$ and such that the leaves $\alpha_n$ of the pants decompositions $\{\omega_{n,i_0}\}$ that intersect $\alpha$ form an infinite family. The lamination $\omega_{\infty,i_0}$ contains $\alpha$ and some leaves spiraling towards $\alpha$. Since $\omega_{n,i_0}$ and $\omega_{n,i_0+1}$ are adjacent vertices, we have $i(\omega_{n,i_0},\omega_{n,i_0+1}) \le 3g+p-2$ for any $n$. It follows that $\omega_{\infty,i_0+1}$ also contains $\alpha$. If, furthermore, $\omega_{\infty, i_0+1}$ contains some leaves spiraling towards $\alpha$, then $\omega_{\infty, i_0+2}$ contains $\alpha$, and the same holds for $\omega_{\infty, i_0+3}$, etc. On the other hand, $\omega_{\infty,q}=Q$ does not contain any leaves spiraling towards $\alpha$. It follows that there is $i_1>i_0$ such that $\omega_{\infty,i_1}$ contains $\alpha$ as an isolated leaf. Then for, $n$ large enough, $\omega_{n,i_1}$ contains $\alpha$, i.e. $\omega_{n,i_1}\subset \CC_\alpha(S)$. Since $\omega_{n,i_0-1}\subset \CC_\alpha(S)$ and $\omega_{n,i_0}\not\subset \CC_\alpha(S)$, the geodesic segment joining $\omega_{n,i_0-1}$ to $\omega_{n,i_1}$ and passing through $\{\omega_{n,i},i_0\leq i\leq i_1-1\}$  is not contained in $\CC_\alpha(S)$. Thus we have proved that $2)$ is not satisfied.\\

\indent
For the other direction, $2)\Rightarrow 1)$, we will show $[\text{not }1)]\Rightarrow  [\text{not } 2)]$. So assume that there is simple closed curve $\alpha\subset S$ such that $\CC_\alpha(S)$ is not totally geodesic. In particular there is a geodesic segment $\omega$ joining two vertices $P,Q\in \CC_\alpha(S)$ such that $\omega\not\subset\CC_\alpha(S)$. Let $T_\alpha:\CC(S)\rightarrow \CP(S)$ be the automorphism induced by the right Dehn twist along $\alpha$. Then $\CC_\alpha(S)$ is exactly the set of points fixed by $T_\alpha$. Since $\omega\not\subset\CC_\alpha(S)$, $\{T_\alpha^n(k),k\in\BZ\}$ is an infinite family of pairwise distinct geodesic segments joining $P$ to $Q$. It follows that $1)$ is not satisfied.
\end{proof}

Combining Proposition \ref{prop:cyrilito}, Theorem \ref{thm:punctures}, and the main results of \cite{APS1,APS2} we obtain the last promised consequence.

\begin{named}{Corollary \ref{geodesics}}
Let $S$ be the six-times punctured sphere. For any $P,Q \in \CP(S)$, there are only finitely many geodesics in $\CP(S)$ between $P$ and $Q$. $\square$
\end{named}

Observe that in the proof above we could have replaced "simple closed curve" by "multi-curve" in $1)$. This, together with Remark \ref{rmk:outside}, implies that the analog of Corollary \ref{geodesics} for the diagonal pants graph is not true.

\end{document}